\documentclass[12pt]{article}
\usepackage{fullpage}
\usepackage{amsmath}
\usepackage{amsthm}
\usepackage{amsfonts}
\usepackage{amssymb}
\usepackage{enumerate}
\usepackage{mathrsfs}
\usepackage{mathtools}
\usepackage{tikz}
\usetikzlibrary{arrows}
\usetikzlibrary{decorations.pathreplacing}
\def\theenumi{\arabic{enumi}}

\def\theenumii{\alph{enumii}}
\def\p@enumii{\theenumi.}

\def\theenumiii{\arabic{enumiii}}
\def\p@enumiii{(\theenumi)(\theenumii)}

\def\p@enumiv{\p@enumiii.\theenumiii}
\newtheorem{construction}{Construction}[section]
\newtheorem{theorem}[construction]{Theorem}

\newtheorem{corollary} [construction]{Corollary}

\newtheorem{definition} [construction]{Definition}

\newtheorem{conditions}[construction]{Conditions}
\newtheorem{lemma} [construction]{Lemma}

  \newcommand{\HWP}{\ensuremath{\mbox{\sf HWP}}}
  \renewcommand{\-}{\ensuremath{\text{--}}}

\providecommand{\keywords}[1]{\textbf{\textit{Keywords---}} #1}
\begin{document}

\title{On the Hamilton-Waterloo problem: the case of two cycles sizes of different parity}
\author{Melissa Keranen\\
Michigan Technological University\\
Department of Mathematical Sciences\\
Houghton, MI 49931, USA\\
\\
Adri\'{a}n Pastine\\
Universidad Nacional de San Luis\\
Departamento de Matem\'{a}tica\\
San Luis, Argentina}

\maketitle
\footnote{email addresses: Melissa Keranen (msjukuri@mtu.edu), Adri\'{a}n Pastine (adrian.pastine.tag@gmail.com)}
\begin{abstract}
The Hamilton-Waterloo problem asks for a decomposition of the complete graph into $r$ copies of a 2-factor $F_{1}$
and $s$ copies of a 2-factor $F_{2}$ such that $r+s=\left\lfloor\frac{v-1}{2}\right\rfloor$.  If $F_{1}$ consists of $m$-cycles
and $F_{2}$ consists of $n$ cycles, then we call such a decomposition a $(m,n)-\HWP(v;r,s)$. The goal is to
find a decomposition for every possible pair $(r,s)$. In this paper, we show that for odd $x$ and $y$,
there is a $(2^kx,y)-\HWP(vm;r,s)$ if
$\gcd(x,y)\geq 3$, $m\geq 3$, and both $x$ and $y$ divide $v$, except possibly when $1\in\{r,s\}$.
\end{abstract}
\keywords{$2$-Factorizations, Hamilton-Waterloo Problem, Oberwolfach Problem, Cycle Decomposition, Resolvable Decompositions}
\section{Introduction}

The Oberwolfach problem asks for a decomposition of the complete graph $K_v$ into $\frac{v-1}{2}$ copies of a $2$-factor $F$. To achieve this decomposition, $v$ needs to be odd, because the vertices must have even degree. The problem with $v$ even asks for a decomposition of $K_v$ into $\frac{v-2}{2}$ copies of a $2$-factor $F$, and one copy of a $1$-factor. The uniform Oberwolfach problem (all cycles of the $2$-factor have the same size) has been completely solved by Alspach and Haagkvist \cite{AH} and Alspach, Schellenberg, Stinson and Wagner \cite{ASSW}. The non-uniform Oberwolfach problem has been studied as well, and a survey of results up to 2006 can be found in \cite{BR}.  Furthermore, one can refer to \cite{BS2,BDD,BDP,RT,T} for more recent results.

In \cite{Liu1} Liu first worked on the generalization of the Oberwolfach problem to equipartite graphs. Here we are seeking to decompose the complete equipartite graph $K_{(m:n)}$ with $n$ partite sets of size $m$ each into $\frac{(n-1)m}{2}$ copies of a $2$-factor $F$. Here $(n-1)m$ has to be even. In \cite{HH} Hoffman and Holliday worked on the equipartite generalization of the Oberwolfach problem when $(n-1)m$ is odd, decomposing into $\frac{(n-1)m-1}{2}$ copies of a $2$-factor $F$, and one copy of a $1$-factor. The uniform Oberwolfach problem over equipartite graphs has since been completely solved by Liu \cite{Liu2} and Hoffman and Holliday \cite{HH}. For the non-uniform case, Bryant, Danziger and Pettersson \cite{BDP} completely solved the case when the $2$-factor is bipartite. In particular, Liu showed the following.

\begin{theorem}\label{equip}
 {\normalfont\cite{Liu2}}  For $m\geq 3$ and $u\geq 2$, $K_{(h:u)}$ has a resolvable $C_m$-factorization if 
 and only if $hu$ is divisible by $m$, $h(u-1)$ is even, $m$ is even if $u=2$, and $(h,u,m) \not \in 
 \{(2,3,3), (6,3,3),(2,6,3),\\(6,2,6)\}$.
\end{theorem}

The Hamilton-Waterloo problem is a variation of the Oberwolfach problem, in which we consider two $2$-factors, $F_1$ and $F_2$. It asks for a factorization of $K_v$ when $v$ is odd or $K_v-I$ ($I$ is a $1$-factor) when $v$ is even into $r$ copies of $F_1$ and $s$ copies of $F_2$ such that $r+s=\left\lfloor\frac{v-1}{2}\right\rfloor$, where $F_1$ and $F_2$ are two $2$-regular graphs on $v$ vertices. Most of the results for the Hamilton-Waterloo problem are uniform, meaning $F_1$ consists of cycles of size $m$ ($C_{m}$-factors), and $F_2$ consists of cycles of size $n$ ($C_{n}$-factors). We refer to a decomposition of $K_v$ into $r$ $C_{m}$-factors and $s$ $C_{n}$-factors as a $(m,n)-\HWP(v;r,s)$.
The case where both $m$ and $n$ are odd positive integers and $v$ is odd is almost completely solved by \cite{BDT,BDT2}; and
if $m$ and $n$ are both even, then the problem again is almost completely solved (see \cite{BD,BDD}). However, if $m$ and
$n$ are of differing parities, then we only have partial results.  Most of the work has been done in the case where one of the cycle sizes is constant. The case of $(m,n)=(3,4)$ is solved in \cite{BB,DQS,OO,WCC}.
Other cases which have been studied include $(m,n)=(3,v)$ \cite{LS}, $(m,n)=(3,3x)$ \cite{AKKPO}, and $(m,n)=(4,n)$ \cite{KO,OO} .

In this paper, we consider the case of $m$ and $n$ being of different parity.  This case has gained attention recently, where it has been shown that the necessary conditions are sufficient for a $(m,n)-\HWP(v;r,s)$ whenever $m \mid n$, $v>6n>36m$, and $s \geq 3$ \cite{BDT3}. We provide a complementary result to this in our main theorem, which covers cases in which $m \nmid n$ and solves a major portion of the problem.

\begin{theorem}
\label{main}
Let $x,y,v,k$ and $m$ be positive integers such that:
\begin{enumerate}[i)]
\item $v,m\geq 3$,
\item $x,y$ are odd,
\item $\gcd(x,y)\geq 3$,
\item $x$ and $y$ divide $v$.
\item $4^k$ divides $v$.
\end{enumerate}
Then there exists a $(2^kx,y)\-\HWP(vm;r,s)$ for every pair $r,s$ with 
$r+s=\left\lfloor(vm-1)/2\right\rfloor$, $r,s\neq 1$.
\end{theorem}


\section{Preliminaries}
Let $G$ and $H$ be multipartite graphs. Then the \textit{partite product} of $G$ and $H$, $G
\otimes   H$ is defined by:
\begin{itemize}
\item $V(G\otimes   H)=\{(g,h,i)| (g,i)\in V(G)$ and $(h,i)\in V(H)\}$.
\item $E(G\otimes   H)=\{\{(g_1,h_1,i),(g_2,h_2,j)\}|\{(g_1,i),(g_2,j)\}\in E(G)$ and $\{(h_1,i),(h_2,j)\}
\in E(H)\}$.
\end{itemize}
where two vertices in $G$ (or $H$) $(g,i),(g,j)$ ($(h,i),(h,j)$) are in the same partite set in $G$ ($H$) 
if and only if $i=j$. The \textit{complete cyclic multipartite graph} $C_{(x:k)}$ is the graph with $k$ partite sets of size 
$x$, where two vertices $(g,i)$ and $(h,j)$ are neighbors if and only if $i-j=\pm 1\pmod{k}$, with 
subtraction being done modulo $k$. The \textit{directed complete cyclic multipartite graph} 
$\overrightarrow{C}_{(x:k)}$ is the graph with $k$ parts of size $x$, with arcs of the form $\big((g,i),
(h,i+1)\big)$ for every $0\leq g,h\leq x-1$, $0\leq i\leq k-1$.

In \cite{KP}, decompositions of $\overrightarrow{C}_{(x:k)}$, $x$ odd, 
into $C_{k}$-factors and $C_{xk}$-factors, and decompositions of $\overrightarrow{C}_{(4:k)}$ into 
$C_{k}$-factors and $C_{2k}$-factors were given. Then by using multivariate bijections, decompositions of 
$\overrightarrow{C}_{(xy:k)}$, $x$ and $y$ odd, into $C_{xk}$-factors and $C_{yk}$-factors were obtained. 
 This was used in conjunction with the following three result to produce the main theorems given in their paper.

\begin{lemma}[\cite{KP}]\label{productofbalanced}
Let $\overrightarrow{G}$ and $\overrightarrow{H}$ be a $\overrightarrow{C}_n$-factor and a $
\overrightarrow{C}_m$-factor of $\overrightarrow{C}_{(x:k)}$ and $\overrightarrow{C}_{(y:k)}$, 
respectively. Then $\overrightarrow{G}\otimes   \overrightarrow{H}$ is a $\overrightarrow{C}_{l}$-factor 
of $\overrightarrow{C}_{(xy:k)}$, where $l=\frac{nm}{gcd(n,m)}$.
\end{lemma}
\begin{theorem}[Distribution,\cite{KP}]\label{distributive}
Let $G=\oplus_i G_i$ and $H=\oplus_j H_j$ be $k$-partite graphs. Then $G\otimes   H=\left(\oplus_i G_i
\right)\otimes  \left(\oplus_j H_j\right)$. Furthermore, the following distributive property holds:
\[
\left(\oplus_i G_i\right)\otimes  \left(\oplus_j H_j\right)=\oplus_i\left(G_i\otimes  \oplus_j H_j\right)=
\oplus_i\oplus_j\left( G_i\otimes   H_j\right)
\]
\end{theorem}

\begin{lemma}[\cite{KP}]\label{cvntokvm}
Let $m$, $x_1,\ldots,x_p$, $y_1,\ldots,y_p$, and $v$ be positive integers. Let $s_1,\ldots,s_{\frac{m-1}
{2}}$ be non-negative integers. Suppose the following conditions are satisfied:
\begin{itemize}
\item There exists a decomposition of $K_m$ into $[n_1,\ldots,n_p]$-factors.
\item For every $1\leq i\leq p$, and for every $1\leq t \leq \frac{m-1}{2}$ there exists a decomposition 
of $C_{(v:n_i)}$ into $s_t$ $C_{x_in_i}$-factors and $r_t$ $C_{y_in_i}$-factors.
\end{itemize}

Let
\[
s=\sum_{t=1}^{\frac{(m-1)}{2}}s_t\text{\hspace{.3in} and \hspace{.3in}}r=\sum_{t=1}^{\frac{(m-1)}{2}}r_t
\]

Then there exists a decomposition of $K_{(v:m)}$ into $s$ $[(x_1n_1)^{\frac{v}{x_1}},\ldots,
(x_pn_p)^{\frac{v}{x_p}}]$-factors and $r$ $[(y_1n_1)^{\frac{v}{y_1}},\ldots,$ $(y_pn_p)^{\frac{v}{y_p}}]$-
factors.
\end{lemma}

Lemmas~\ref{productofbalanced} and ~\ref{cvntokvm} and Theorem~\ref{distributive} will be employed in this paper as well to produce
the decompositions we are interested in.
In Section \ref{section even case} we give decompositions of $\overrightarrow{C}_{(4^k:n)}$ into $
\overrightarrow{C}_{2^kn}$-factors and $\overrightarrow{C}_{n}$-factors. In Section \ref{section 
multivariate bijections} we use multivariate bijections to give decompositions of 
$\overrightarrow{C}_{(4^kxy:n)}$ into $\overrightarrow{C}_{2^kxk}$-factors and 
$\overrightarrow{C}_{yk}$-factors. Then in Section \ref{section main results}, we
use these decompositions to prove our main results.

\section{Equipartite Decompositions}\label{section even case}
We will start decomposing $\overrightarrow{C}_{(4^k:3)}$ into $\overrightarrow{C}_{2^{k}3}$-factors and 
$\overrightarrow{C}_{3}$-factors.
Label the vertices in each partite set of $\overrightarrow{C}_{(4^k:3)}$ with the elements of the quotient 
ring $\mathfrak{R}=\mathbb{Z}_{2^{k}}[x]/\left\langle x^2+x+1\right\rangle$. Because the elements of 
$\mathfrak{R}$ are of the form $ax+b$, with $a,b\in\mathbb{Z}_{2^{k}}$, there are $4^k$ of them.
For each $\alpha \in \mathfrak{R}$ let $f_{\alpha}(y)=xy+\alpha$.

\begin{lemma}\label{lemmapropertiesoffalpha}
The functions $f_{\alpha}$ are bijections, with $f^3_{\alpha}(y)=y$.
\end{lemma}

\begin{proof}
We will first show that $f_{0}$ is a bijection.
Notice that the element $x\in \mathfrak{R}$ is a unit, because 
\begin{align*}
 x^2+x+1&=0,&\text{and}& &x(-x-1)&=-x^2-x=1.\\
\end{align*}
Therefore, $xy_1=xy_2$ if and only if $y_1=y_2$, and so $f_{0}$ is a bijection. Because every element has 
an additive inverse we get that $f_{\alpha}$ is a bijection.
We will see now that $f^3_{\alpha}(y)=y$:

\begin{align*}
f_{\alpha}(y)&=xy+\alpha\\
f^2_{\alpha}(y)&=x(xy+\alpha)+\alpha\\
&=x^2y+x\alpha+\alpha\\
&=-xy-y+x\alpha+\alpha\\
f^3_{\alpha}(y)&=x(-xy-y+x\alpha+\alpha)+\alpha\\
&=-x^2y-xy+x^2\alpha+x\alpha+\alpha\\
&=xy+y-xy+(x^2+x+1)\alpha\\
&=y+0\alpha\\
&=y
\end{align*}

Therefore $f_{\alpha}$ is a bijection and $f^3_{\alpha}(y)=y$.
\end{proof}

Let $T_{4^k}(\alpha)$ be the subgraph of $\overrightarrow{C}_{(4^k:3)}$ where each element $y$ in a 
partite set is connected to the element $f_{\alpha}(y)$ in the next partite set. Because of Lemma 
\ref{lemmapropertiesoffalpha}, $T_{4^k}(\alpha)$ is a $\overrightarrow{C}_3$-factor ($f^3_{\alpha}(y)=y$)
and given two different elements, $\alpha$ and $\beta$, $T_{4^k}(\alpha)$ and $T_{4^k}(\beta)$ are disjoint.
Therefore $\overrightarrow{C}_{(4^k:3)}=\bigoplus_{\alpha\in \mathfrak{R}}T_{4^k}(\alpha)$ is a $
\overrightarrow{C}_3$-factorization of $\overrightarrow{C}_{(4^k:3)}$.

Let $H_{4^k}(\alpha,\beta)$ be the subgraph of $\overrightarrow{C}_{(4^k:3)}$ where each element $y$ of 
the first and second partite sets are connected to the element $f_{\alpha}(y)$ of the second and third 
partite sets respectively, and each element $y$ of the third partite set is connected to the element 
$f_{\beta}(y)$ of the first partite set. The following result is easy to see:

\begin{lemma}\label{fromTtoH}
Let $\phi$ be a permutation of $\mathfrak{R}$, then
\[
\overrightarrow{C}_{(4^k:3)}=\bigoplus_{\alpha\in \mathfrak{R}}T_{4^k}(\alpha)=\bigoplus_{\alpha\in 
\mathfrak{R}}H_{4^k}(\alpha,\phi(\alpha))
\]
and $\bigoplus_{\alpha\in \mathfrak{R}}H_{4^k}(\alpha,\phi(\alpha))$ is a decomposition of $
\overrightarrow{C}_{(4^k:3)}$.
\end{lemma}

\begin{proof}
The first equality is true by the discussion preceding this lemma. The second equality is true because of 
$\phi$ being a permutation (each edge gets used once).
\end{proof}

Consider $y_1\in \mathfrak{R}$. Let $y_2=f_{\alpha}(y_1)$, $y_3=f_{\alpha}(y_2)$, and $y_4=f_{\beta}(y_3)
$. Then because $f_{\alpha}^3(y_1)=y_1$, we have $xy_3+\alpha=f_{\alpha}(y_3)=f_{\alpha}^3(y_1)=y_1$. We 
also have $f_{\beta}(y_3)=xy_3+\beta=y_4$. So $y_1-y_4=\alpha-\beta$. Therefore, if $\alpha-\beta\in\{\pm 
1,\pm x,\pm x\pm 1\}$, then  $H_{4^k}(\alpha,\beta)$ consists of directed cycles of length $3\cdot 2^{k}$ 
because $\alpha-\beta$ has additive order $2^k$ in $\mathfrak{R}$. Hence if $\phi$ has $r$ fixed points 
and for every non-fixed point $\alpha$ we have $\alpha-\phi(\alpha)\in\{\pm 1, \pm x, \pm x \pm 1\}$, we 
obtain a decomposition of $\overrightarrow{C}_{(4^k:3)}$ into $r$ $\overrightarrow{C}_3$-factors and $4^k-
r$ $\overrightarrow{C}_{3\cdot2^k}$-factors. If $\alpha=\beta$, then because $H_{4^k}(\alpha,\alpha)=T_{4^k}(\alpha)$, we have that $H_{4^k}(\alpha,\beta)$ is a $\overrightarrow{C}_3$-factor of $\overrightarrow{C}_{(4^k:3)}$. Thus we have the following lemma.

\begin{lemma}\label{3h4k}
If $\alpha=\beta$, then $H_{4^k}(\alpha,\beta)$ is a $\overrightarrow{C}_{3}$-factor of 
$\overrightarrow{C}_{(4^k:3)}$. 
If $\alpha-\beta\in\{\pm 1,\pm x,\pm x\pm 1\}$, then $H_{4^k}(\alpha,\beta)$ is a 
$\overrightarrow{C}_{3\cdot 2^k}$-factor of $\overrightarrow{C}_{(4^k:3)}$. 
\end{lemma}

Let $\pi:\mathfrak{R}\rightarrow \mathbb{Z}_{2^{k-1}}\times \mathbb{Z}_{2^{k-1}}\times \mathbb{Z}_{4}$ 
defined as:
\begin{center}
\[
\pi(a+bx)\coloneqq\left\lbrace\begin{array}{ccc}
(\lfloor a/2\rfloor,\lfloor b/2\rfloor,0)&\text{ if }& \text{ $a$ and $b$ are even}\\
(\lfloor a/2\rfloor,\lfloor b/2\rfloor,1)&\text{ if }& \text{ $a$ is odd and $b$ are even}\\
(\lfloor a/2\rfloor,\lfloor b/2\rfloor,2)&\text{ if }& \text{ $a$ is even and $b$ is odd}\\
(\lfloor a/2\rfloor,\lfloor b/2\rfloor,3)&\text{ if }& \text{ $a$ and $b$ are odd}\\
\end{array}\right.
\]
\end{center}
Notice that $\pi$ is a bijection, and that if $\rho$ is a permutation on $\mathbb{Z}_{2^{k-1}}\times 
\mathbb{Z}_{2^{k-1}}\times \mathbb{Z}_{4}$ that fixes the first two coordinates and has $r$ fixed points, 
then $\phi=\pi^{-1}(\rho(\pi))$ is a permutation of $\mathfrak{R}$ that has $r$ fixed points and for every 
non-fixed point $\alpha$ we have $\alpha-\phi(\alpha)\in\{\pm 1, \pm x, \pm x \pm 1\}$.

Because we are asking for $\rho$ to fix the first two coordinates, finding the necessary function is
similar to finding $4^k$ permutations $\rho_{\dot{a},\dot{b}}$ of $\mathbb{Z}_4$ with the necessary number of
fixed points. Then we have $\rho(\dot{a},\dot{b},i)=(\dot{a},\dot{b},\rho_{\dot{a},\dot{b}}(i))$.

\begin{lemma}\label{decomCv3}
Let $r\neq 4^k-1$, then there is a decomposition of $\overrightarrow{C}_{(4^k:3)}$ into $r$ 
$\overrightarrow{C}_3$-factors and $s=4^k-r$ $C_{3*2^k}$-factors.
\end{lemma}
\begin{proof}
Let $r_{\dot{a},\dot{b}}$, $0\leq \dot{a},\dot{b}\leq 2^{k-1}-1$ 
be such that $r_{\dot{a},\dot{b}}\in\{0,1,2,4\}$, 
$\displaystyle\sum_{\dot{a},\dot{b}}r_{\dot{a},\dot{b}}=r$. 
Let $\rho_{\dot{a},\dot{b}}$ be a permutation on $\mathbb{Z}_4$ with $r_{\dot{a},\dot{b}}$ fixed points.
Let $\rho(\dot{a},\dot{b},i)=(\dot{a},\dot{b},\rho_{\dot{a},\dot{b}}(i))$, and $\phi=\pi^{-1}(\rho(\pi))$. 
Then the decomposition is given by
\[
\overrightarrow{C}_{(4^k:3)}=\bigoplus_{\alpha \in \mathfrak{R}}H_{4^k}(\alpha,\phi(\alpha))
\]
\end{proof}

Given an $n$-partite graph $G$, with parts $G_0,G_1,\ldots,G_{n-1}$, let $F_h(G)$ be the subgraph of $G$ 
that contains only the edges between parts $h-1$ and $h$. In particular
$F_n(G)$ contains the edges between $G_{n-1}$ and $G_0$.

\begin{theorem}\label{4ktheorem}
Let $r\neq 4^k-1$ and $n\geq 5$. Then there is a decomposition of $\overrightarrow{C}_{(4^k:n)}$ into $r$ 
$\overrightarrow{C}_n$-factors and $s=4^k-r$ $\overrightarrow{C}_{n\cdot 2^k}$-factors.
\end{theorem}
\begin{proof}
First suppose $n$ is odd, and let $G_0,G_1,\ldots,G_{n-1}$ be the partite sets. Let $T_{4^k}(\alpha)$ be 
the subgraph of $\overrightarrow{C}_{(4^k:n)}$ having the following edges:
\begin{itemize}
\item A vertex $y$ in $G_i$ is adjacent to $y+\alpha$ in $G_{i+1}$ if $i<n-3$ is odd;
\item A vertex $y$ in $G_i$ is adjacent to $y-\alpha$ in $G_{i+1}$ if $i<n-3$ is even;
\item A vertex $y$ in $G_i$ is adjacent to $f_{\alpha}(y)$ in $G_{i+1}$ if $i\geq n-3$.
\end{itemize}
Notice that a directed cycle which starts at vertex $y$ in $G_0$ contains the vertex $y$ in $G_{n-3}$.
Because $f_{\alpha}^{3}(y)=y$ by Lemma \ref{lemmapropertiesoffalpha}, we have that $T_{4^k}(\alpha)$ is a 
$\overrightarrow{C}_n$-factor of $\overrightarrow{C}_{(4^k:n)}$. Furthermore, because of the way we
described the edges in $T_{4^k}(\alpha)$, we can identify the partite set $G_{n-3}$ with $G_0$. 
Then any $3$-cycle in $\overrightarrow{C}_{(4^k:3)}$ is equivalent to an $n$-cycle in 
$\overrightarrow{C}_{(4^k:n)}$. Now let 
$H_{4^k}(\alpha,\beta)=T_{4^k}(\alpha)\oplus F_n(T_{4^k}(\alpha))\oplus F_n(T_{4^k}(\beta))$, 
where the arcs $F_n(T_{4^k}(\gamma))$ consist of the arcs in $T_{4^k}(\gamma)$ from $G_{n-1}$ to $G_0$. 
Again, we may identify $G_{n-3}$ with $G_0$, so a directed cycle of length $3\cdot 2^k$ in 
$\overrightarrow{C}_{(4^k:3)}$ is now equivalent to a directed cycle of length $n\cdot 2^k$ in 
$\overrightarrow{C}_{(4^k:n)}$. So by Lemma \ref{decomCv3}, since there is a decomposition of 
$\overrightarrow{C}_{(4^k:3)}$ into $r$ $\overrightarrow{C}_3$-factors and $s=4^k-r$ 
$\overrightarrow{C}_{3\cdot 2^k}$-factors, this is equivalent to a decomposition of 
$\overrightarrow{C}_{(4^k:n)}$ into $r$ $\overrightarrow{C}_n$-factors and $s=4^k-r$ 
$\overrightarrow{C}_{n\cdot 2^k}$-factors.

If $n\geq 6$ is even, then let $T_{4^k}(\alpha)$ be the subgraph of $\overrightarrow{C}_{(4^k:n)}$ having 
the following edges:
\begin{itemize}
\item A vertex $y$ in $G_i$ is adjacent to $y+\alpha$ in $G_{i+1}$ if $i<n-6$ is odd;
\item A vertex $y$ in $G_i$ is adjacent to $y-\alpha$ in $G_{i+1}$ if $i<n-6$ is even;
\item A vertex $y$ in $G_i$ is adjacent to $f_{\alpha}(y)$ in $G_{i+1}$ if $i\geq n-6$.
\end{itemize}
Now a cycle that starts at vertex $y$ in $G_0$ contains the vertex $y$ in $G_{n-6}$, and because 
$f_{\alpha}^3(y)=y$, the cycle also contains vertex $y$ in $G_{n-3}$. We may now apply the same arguments 
as in the case with $n$ odd to obtain the result.
\end{proof}

If we define $T_{4^k}(\alpha)$ and $H_{4^k}(\alpha,\beta)$ as in the proof of Theorem \ref{4ktheorem}, then we obtain the following corollary by applying Lemma \ref{3h4k}.
\begin{corollary}\label{nh4k}
If $\alpha=\beta$, then $H_{4^k}(\alpha,\beta)$ is a $\overrightarrow{C}_{n}$-factor of 
$\overrightarrow{C}_{(4^k:n)}$. 
If $\alpha-\beta\in\{\pm 1,\pm x,\pm x\pm 1\}$, then $H_{4^k}(\alpha,\beta)$ is a 
$\overrightarrow{C}_{2^kn}$-factor of $\overrightarrow{C}_{(4^k:n)}$. 
\end{corollary}

\section{Multivariate Bijections}\label{section multivariate bijections}

When $x$ is odd, the graphs $T_x(i)$ and $H_x(i,i')$ were defined in \cite{KP} as follows.

$T_x(i)$ is the subgraph of $\overrightarrow{C}_{(x:n)}$ obtained by taking differences:
\begin{itemize}
\item $2^{e_j}i$ between $G_{j-1}$ and $G_{j}$ for $1\leq j \leq k$;
\item $-2i$ between $G_j$ and $G_{j+1}$ for $k\leq j \leq 2k-2$;
\item $-i$ between $G_j$ and $G_{j+1}$ for $2k-1\leq j \leq n-2$;
\item $-i$ between $G_{n-1}$ and $G_{0}$.
\end{itemize}
Then $H_x(i,i')=T_x(i)\oplus F_n(T_x(i))\oplus F_n(T_x(i'))$, and notice that $H_x(i,i)=T_x(i)$.

\begin{lemma}[\cite{KP}]\label{bbb}
$T_x(i)$ is a $\overrightarrow{C}_n$-factor for any $i$.
\end{lemma}
\begin{lemma}[\cite{KP}]\label{HisforHamilton}
If $gcd(x,i-s)=1$ then $H_x(i,s)$ is a directed Hamiltonian cycle.
\end{lemma}

Given $x$, $y$ and $k$, positive integers with $x,y$ odd,  we will use ideas similar to Section $7$ of $\cite{KP}$ to obtain decompositions of
 $\overrightarrow{C}_{(4^kxy:n)}$ into $\overrightarrow{C}_{2^{k}xn}$-factors and $
\overrightarrow{C}_{yn}$-factors.

\begin{definition}
Let $x$ and $y$ be odd. Define $T_{(4^kxy)}(\alpha,i,j)$ to be the directed subgraph of $
\overrightarrow{C}_{(4^kxy:n)}$ obtained by taking $T_{(4^kxy)}(\alpha,i,j)=T_{(4^k)}(\alpha)\otimes 
T_x(i)\otimes T_y(j)$. We also define 
\[
H_{(4^kxy)}(\alpha,i,j)(\beta,i',j')=T_{(4^kxy)}(\alpha,i,j)\oplus F_n(T_{(4^kxy)}(\alpha,i,j))\oplus 
F_n(T_{(4^kxy)}(\beta,i',j').\]
\end{definition}
This means that $H_{(4^kxy)}(\alpha,i,j)(\beta,i',j')$ is the directed graph obtained by taking the arcs 
of $T_{(4^kxy)}(\alpha,i,j)$ between parts $t$ and $t+1$ for $0\leq t\leq n-2$, and the arcs between parts 
$n-1$ and $0$ from $T_{(4^kxy)}(\beta,i',j')$.

\begin{lemma}
Let $x$, $y$ and $n$ be odd. Then
\[
H_{(4^kxy)}(\alpha,i,j)(\beta,i',j')=H_{(4^k)}(\alpha,\beta)\otimes H_x(i,i')\otimes   H_y(j,j').
\]
\end{lemma}

\begin{proof}
Notice that 
\[
F_n(T_{(4^kxy)}(\alpha,i,j))=F_n(T_{4^k}(\alpha)\otimes   T_x(i)\otimes T_y(j))=F_n(T_{4^{k}}(\alpha))\otimes   F_n(T_x(i))\otimes F_n(T_y(j))).
\]
Notice also that 
\begin{align*}
F_n(T_{4^k}(\alpha)\otimes   T_x(i)\otimes T_y(j))=&F_n(T_{4^k}(\alpha))\otimes  T_x(i)\otimes T_y(j)\\
=&T_{4^k}(\alpha)\otimes F_n(T_x(i))\otimes T_y(j)\\
=&T_{4^k}(\alpha)\otimes  T_x(i)\otimes F_n(T_y(j)).
\end{align*}

Then we have 
\begin{align*}
H_{4^k}(\alpha,\beta)\otimes H_x(i,i')\otimes H_y(j,j') =&\left[T_{4^k}(\alpha)\oplus F_n(T_{4^k}(\alpha))
\oplus F_n(T_{4^k}(\beta))\right]\\
& \otimes \left[T_x(i)\oplus F_n(T_x(i))\oplus F_n(T_x(i'))\right] \otimes \left[T_y(j)\oplus F_n(T_y(j))\oplus F_n(T_y(j'))\right] \\
=&\left[T_{4^k}(\alpha)\otimes T_x(i)\otimes T_y(j)\right]\oplus \left[F_n(T_{4^k}(\alpha))\otimes 
F_n(T_x(i))\otimes F_n(T_y(j))\right]\\
&\oplus \left[F_n(T_{4^k}(\beta))
\otimes F_n(T_x(i'))\otimes F_n(T_y(j'))\right]\\
=&T_{(4^kxy)}(\alpha,i,j)\oplus F_n(T_{(4^kxy)}(\alpha,i,j))\oplus F_n( T_{(4^kxy)}(\beta,i',j'))\\
=&H_{(4^kxy)}(\alpha,i,j)(\beta,i',j').
\end{align*}

\end{proof}

\begin{lemma}
Let $\varphi$ be a permutation of $\mathfrak{R}\times \mathbb{Z}_x\times \mathbb{Z}_y$. Then
\[
\overrightarrow{C}_{(4^kxy:n)}=\bigoplus_{(\alpha,i,j)}H_{(4^kxy)}(\alpha,i,j)\varphi(\alpha,i,j).\]
\end{lemma}

\begin{proof}
From Theorem~\ref{distributive}, we know that 
\[\overrightarrow{C}_{(4^kxy:n)}=\overrightarrow{C}_{(4^k:n)}\otimes   
\overrightarrow{C}_{(x:n)}\otimes \overrightarrow{C}_{(y:n)}=
\left(\bigoplus_{\alpha} T_{4^k}(\alpha)\right)\otimes  \left(\bigoplus_{i}
T_{x}(i)\right)\otimes \left(\bigoplus_j T_{y}(j)\right).\] 
By the definition of $T_{(4^kxy)}(\alpha,i,j)$
we get 
\[
\overrightarrow{C}_{(4^kxy:n)}=\bigoplus_{(\alpha,i,j)}T_{(4^kxy)}(\alpha,i,j).\]

We also have
\[
\bigoplus_{(\alpha,i,j)}T_{(4^kxy)}(\alpha,i,j)=
\bigoplus_{(\alpha,i,j)}H_{(4^kxy)}(\alpha,i,j)\varphi(\alpha,i,j).
\]

Combining both we get:
\[
\overrightarrow{C}_{(4^kxy:n)}=\bigoplus_{(\alpha,i,j)}H_{(4^kxy)}(\alpha,i,j)\varphi(\alpha,i,j),
\]
as we wanted to prove.
\end{proof}

Because we are dealing with bijections on cartesian products of sets, we introduce the 
following notation. If $\varphi$ is a bijection of $S_1\times S_2\times \ldots \times S_n$, then $
\varphi_i(s_1,\ldots,s_n)$ is the $i$-th coordinate of $\varphi(s_1,\ldots,s_n)$.
Because $H_{4^kxy}(\alpha,i,j)\varphi(\alpha,i,j)=H_{4^k}(\alpha,\varphi_1(\alpha,i,j))\otimes 
H_x(i,\varphi_2(\alpha,i,j))\otimes H_y(j,\varphi_3(\alpha,i,j))$, 
we have:

\begin{itemize}
\item If $\alpha=\varphi_1(\alpha,i,j)$, $i=\varphi_2(\alpha,i,j)$ and 
$\gcd(y,j-\varphi_3(\alpha,i,j))=1$ (or $y=1$), then  
$H_{4^k}(\alpha,\varphi_1(\alpha,i,j))$ is a 
$\overrightarrow{C}_n$-factor of $\overrightarrow{C}_{(4^k:n)}$ by Corollary \ref{nh4k},
$H_x(i,\varphi_2(\alpha,i,j))$ is a 
$\overrightarrow{C}_n$-factor of $\overrightarrow{C}_{(x:n)}$ by Lemma \ref{bbb}, $H_y(j,
\varphi_3(\alpha,i,j))$ is a 
$\overrightarrow{C}_{yn}$-factor of $\overrightarrow{C}_{(y:n)}$ by Lemma \ref{HisforHamilton}, 
and $H_{4^kxy}(\alpha,i,j)\varphi(\alpha,i,j)$ 
is a $\overrightarrow{C}_{yn}$-factor of $\overrightarrow{C}_{(4^kxy:n)}$ 
by Lemma \ref{productofbalanced}.

\item If $\alpha-\varphi_1(\alpha,i,j)\in\{\pm 1,\pm x,\pm x\pm 1\}$, 
$\gcd(x,i-\varphi_2(\alpha,i,j))=1$ (or $x=1$), and $j=\varphi_3(\alpha,i,j)$ 
then  $H_{4^k}(\alpha,\varphi_1(\alpha,i))$ is a 
$\overrightarrow{C}_{2^kn}$-factor of $\overrightarrow{C}_{(4^k:n)}$  by Corollary \ref{nh4k}, $H_x(i,
\varphi_2(\alpha,i))$ is a $
\overrightarrow{C}_{xn}$-factor of $\overrightarrow{C}_{(x:n)}$ by Lemma \ref{HisforHamilton}, 
$H_y(j,\varphi_3(\alpha,i,k))$ is a $\overrightarrow{C}_{n}$-factor of $\overrightarrow{C}_{(y:n)}$ by Lemma~\ref{bbb} and 
$H_{4^kxy}(\alpha,i,j)\varphi(\alpha,i,j)$ 
is a $\overrightarrow{C}_{2^kxn}$-factor of $\overrightarrow{C}_{(4^kxy:n)}$ by Lemma 
\ref{productofbalanced}.
\end{itemize}

For a decomposition of $\overrightarrow{C}_{(4^kxy:n)}$ into $\overrightarrow{C}_{2^kxn}$-factors 
and $\overrightarrow{C}_{yn}$-factors we need a bijection $\varphi$ of $\mathfrak{R}\times \mathbb{Z}_x
\times \mathbb{Z}_y$ that satisfies:
\begin{conditions}\label{2xnandyn}
\begin{enumerate}[a)]
\item For all $(\alpha,i,j)$, $\alpha-\varphi_1(\alpha,i,j)\in\{\pm 1,\pm x,\pm x\pm 1\}$ or 
$\alpha=\varphi_1(\alpha,i,j)$.
\item If $\alpha=\varphi_1(\alpha,i,j)$, then $i=\varphi_2(\alpha,i,j)$ and 
$\gcd(y,j-\varphi_{3}(\alpha,i,j))=1$ (or $y=1$).
\item If $\alpha-\varphi_1(\alpha,i,j)\in\{\pm 1,\pm x,\pm x\pm 1\}$, then 
$\gcd(x,i-\varphi_2(\alpha,i,j))=1$ (or $x=1$), and $j=\varphi_3(\alpha,i,j)$.
\end{enumerate}
\end{conditions}

Define the bijection $\theta:\mathfrak{R}\times \mathbb{Z}_x\times\mathbb{Z}_y\rightarrow 
\mathbb{Z}_{2^{k-1}}\times \mathbb{Z}_{2^{k-1}}\times \mathbb{Z}_{4} \times 
\mathbb{Z}_x\times\mathbb{Z}_y$ by $\theta(\alpha,i,j)=(\pi(\alpha),i,j)$. Then finding such a function 
$\varphi$ is equivalent to finding $4^{k-1}$ functions $\lambda^{(\dot{a},\dot{b})}$ of 
$\mathbb{Z}_4\times\mathbb{Z}_x\times \mathbb{Z}_y$ satisfying Conditions \ref{2xnandynproj} and
\[
\sum_{\dot{a},\dot{b}} |\{(\gamma,i,j)|\gamma=\lambda^{(\dot{a},\dot{b})}_1(\gamma,i,j)\}|=
|\{(\gamma,i,j)|\gamma=\varphi_1(\gamma,i,j)\}|.
\]

\begin{conditions}\label{2xnandynproj}
\begin{enumerate}[a)]
\item If $\gamma=\lambda^{(\dot{a},\dot{b})}_1(\gamma,i,j)$, then 
$i=\lambda^{(\dot{a},\dot{b})}_2(\gamma,i,j)$ and $\gcd(y,j-\lambda^{(\dot{a},\dot{b})}_3(\gamma,i,j))=1$
(or $y=1$).
\item If $\gamma\neq \lambda^{(\dot{a},\dot{b})}_1(\gamma,i,j)$, then 
$\gcd(x,i-\lambda^{(\dot{a},\dot{b})}_{2}(\gamma,i,j))=1$ (or $x=1$) and
$j=\lambda^{(\dot{a},\dot{b})}_3(\gamma,i,j)$.
\end{enumerate}
\end{conditions}

The existence of such bijections was shown in Lemma $7.17$ of $\cite{KP}$ (Lemma $7.11$ if $y=1$, Lemma $7.12$ if $x=1$). 
 Hence we have:

\begin{lemma}\label{2xnandynlemma}
Let $s_p\in \{0,2,3,\ldots,4^kxy-3,4^kxy-2,4^kxy\}$. Then there exists a decomposition of $
\overrightarrow{C}_{(4^kxy:n)}$ into $s_p$ $\overrightarrow{C}_{2^{k}xn}$-factors and $r_p=4^kxy-s_p$ $
\overrightarrow{C}_{yn}$-factors.
\end{lemma}
\begin{proof}
Let $s_p=\sum_{\dot{a},\dot{b}}s_{\dot{a},\dot{b}}$, with $s_{\dot{a},\dot{b}}\in\{0,2,\ldots,4xy\}$. By 
Lemma $7.17$ of $\cite{KP}$ (Lemma $7.11$ if $y=1$, Lemma $7.12$ if $x=1$), for each pair 
$\dot{a},\dot{b}$ there exists a permutation $\lambda^{(\dot{a},
\dot{b})}$ of 
$\mathbb{Z}_4\times \mathbb{Z}_x\times \mathbb{Z}_y$ satisfying Conditions \ref{2xnandynproj} such that
\[
\gamma=\lambda^{(\dot{a},\dot{b})}_1(\gamma,i,j)
\]
holds for $s_{\dot{a},\dot{b}}$ elements $(\gamma,i,j)$.

Then $\lambda(\dot{a},\dot{b},\gamma,i,j)=(\dot{a},\dot{b},\lambda^{(\dot{a},\dot{b})}(\gamma,i,j))$ is a 
permutation of $\mathbb{Z}_{2^{k-1}}\times 
\mathbb{Z}_{2^{k-1}}\times \mathbb{Z}_{4} \times \mathbb{Z}_x \times \mathbb{Z}_y$, 
and $\varphi=\theta^{-1}\lambda\theta$ is a 
permutation of $\mathfrak{R}\times \mathbb{Z}_x\times \mathbb{Z}_y$ 
satisfying Conditions \ref{2xnandyn}, with 
\[
s_p=\sum_{\dot{a},\dot{b}}(s_{\dot{a},\dot{b}})
\]
pairs satisfying $\alpha=\varphi_1(\alpha,i,j)$.
\end{proof}

\section{Main Results}\label{section main results}
The complete solution to the uniform case of the Oberwolfach problem will be vital to the proof of our main result.
\begin{theorem}[\cite{AH,ASSW,HS,RW}]
\label{OP}
$K_v$ can be decomposed into $C_m$-factors (and a $1$-factor is $v$ is even) if and
only if $v \equiv 0 \pmod{m}$, $(v,m) \not = (6,3)$ and $(v,m) \not = (12,3)$.
\end{theorem}

We now apply the results from Section~\ref{section multivariate bijections} to produce the following important result for the
uniform equipartite version of the Hamilton-Waterloo problem where the two factor types consist
of cycle sizes of distinct parities.

\begin{theorem}\label{lemmausingliu}
Let $x,y,z,v,m,k$ be positive integers $v,m,k\geq 3$ satisfying the following:
\begin{enumerate}[i)]
\item $v,m\geq 3$,
\item $k\geq 2$,
\item $x,y,z$ odd,
\item $z\geq 3$,
\item $\gcd(x,y)=1$,
\item $vm\equiv 0 \pmod{4^kxyz}$, $v\equiv 0\pmod{4^kxy}$,
\item $\frac{v(m-1)}{4^kxy}$ is even,
\item $\left(\frac{v}{4^kxy},m,z\right) \not \in \{(2,3,3),(6,3,3),(2,6,3),(6,2,6)\}$
\end{enumerate}
then there is a decomposition of $K_{(v:m)}$ into $r$ $C_{2^kxz}$-factors and $s$ $C_{yz}$-factors, for 
any $s,r\neq 1$.
\end{theorem}
\begin{proof}
Let $v_1=v/4^kxy$. Consider $K_{(v_1:m)}$. Item $vi$ ensures that $z$ divides $v_{1}m$; and items
$vii$, $i$, and $viii$ give us $v_1(m-1)$ is even, 
$m\neq 2$, and $\left(\frac{v}{4^kxy},m,z\right) \not \in \{(2,3,3)$, $(6,3,3), (2,6,3), (6,2,6)\}$.
Thus by Theorem \ref{equip} there is a decomposition of 
$K_{(v_1:m)}$ into $C_{z}$-factors.

Give weight $4^kxy$ to the vertices in $K_{(v_1:m)}$, which yields $K_{(v:m)}$. 
Even more, each $C_{z}$-factor becomes a copy of $\frac{v_1m}{z}C_{(4^kxy:z)}$. 
By Lemma \ref{2xnandynlemma}, we have that each $\frac{v_1m}{z}C_{(4^kxy:z)}$ 
can be decomposed into $r_p$ $C_{2^{k}xz}$-factors and 
$s_p$ $C_{yz}$-factors as long as $r_p,s_p\neq 1$. Choosing $s_p$ such that $\sum_p s_p=s$ and
$s_p,r_p\neq 1$, provides a decomposition of $K_{(v:m)}$ into $r$ $C_{2^{k}xz}$-factors and 
$s$ $C_{yz}$-factors by Lemma~\ref{cvntokvm}
\end{proof}

The next lemma, given in \cite{KP} shows how to find solutions to the Hamilton-Waterloo problems by
combining solutions for the problem on
complete graphs and solutions for the problem on equipartite graphs.

\begin{lemma}[\cite{KP}]\label{buildcompletegraphnonuniform}
Let $m$ and $v$ be positive integers. Let $F_1$ and $F_2$ be two $2$-factors on $vm$ vertices. Suppose the 
following conditions are satisfied:
\begin{itemize}
\item There exists a decomposition of $K_{(v:m)}$ into $s_\alpha$ copies of $F_1$ and $r_\alpha$ copies of 
$F_2$.
\item There exists a decomposition of $mK_v$ into $s_\beta$ copies of $F_1$ and $r_\beta$ copies of 
$F_2$.
\end{itemize}

Then there exists a decomposition of $K_{vm}$ into $s=s_\alpha+s_\beta$ copies of $F_1$ and $r=r_\alpha+r_
\beta$ copies of $F_2$.
\end{lemma}

We are now in a position to provide a proof of the main theorem.

\begin{theorem}
Let $x,y,v,k$ and $m$ be positive integers such that:
\begin{enumerate}[i)]
\item $v,m\geq 3$,
\item $x,y$ are odd,
\item $\gcd(x,y)\geq 3$,
\item $x$ and $y$ divide $v$.
\item $4^k$ divides $v$.
\end{enumerate}
Then there exists a $(2^kx,y)\-\HWP(vm;r,s)$ for every pair $r,s$ with 
$r+s=\left\lfloor(vm-1)/2\right\rfloor$, $r,s\neq 1$.
\end{theorem}

\begin{proof}
Let $r$ and $s$ be positive integers with $r+s=\left\lfloor(vm-1)/2\right\rfloor$ and $r,s\neq 1$.
Write $r=r_{\alpha}+r_{\beta}$ and $s=s_{\alpha}+s_{\beta}$, where $r_{\alpha},r_{\beta},s_{\alpha},s_{\beta}$ are positive
integers that satisfy $r_\alpha,s_\alpha\neq 1$, $r_\alpha+s_\alpha=v(m-1)/2$, $r_\beta+s_\beta=\left\lfloor(v-1)/2\right\rfloor$,
and $r_\beta,s_\beta\in\{0,\left\lfloor(v-1)/2\right\rfloor\}$.

Start by decomposing $K_{vm}$ into $K_{(v:m)}\oplus mK_v$.  
Let $z=\gcd(x,y)$, $x_1=x/z$, $y_1=y/z$. By Theorem \ref{lemmausingliu} there is a decomposition of 
$K_{(v:m)}$ into $r_\alpha$ $C_{2^kx_1z}$-factors and $s_\alpha$ $C_{y_1z}$-factors. This is a 
decomposition of $K_{(v:m)}$ into  $r_\alpha$ $C_{2^kx}$-factors and $s_\alpha$ $C_{y}$-factors.
By Theorem \ref{OP} there is a decomposition of $mK_v$ into $r_\beta$ $C_{2^kx}$-factors and 
$s_\beta$ $C_{y}$-factors. Lemma~\ref{buildcompletegraphnonuniform} shows that 
all of this together yields a decomposition of $K_{vm}$ into $r$ $C_{x}$-factors and $s$ $C_{y}$-factors.
\end{proof}

\end{document}